\documentclass[12pt]{amsart} 
\usepackage{amsmath,amsthm,amssymb,mathtools,marvosym}
\usepackage[english]{babel}
\usepackage[autostyle]{csquotes}
\newcommand*{\field}[1]{\mathbb{#1}}%
\newdimen\slantmathcorr
\def\oversl#1{
\setbox0=\hbox{$#1$}
\slantmathcorr=\wd0
\hskip 0.2\slantmathcorr \overline{\hbox to 0.8\wd0{
\vphantom{\hbox{$#1$}}}}
\hskip-\wd0\hbox{$#1$}
}

\newtheorem{theorem}{Theorem}[section]
\newtheorem{Lemma}[theorem]{Lemma}
\newtheorem{Proposition}[theorem]{Proposition}
\newtheorem{Corollary}[theorem]{Corollary}
\newtheorem{Definition}[theorem]{Definition}
\newtheorem{Example}[theorem]{Example}
\newtheorem{Remark}[theorem]{Remark}

\begin{document} 
\title[Specification for group actions on Uniform spaces]{Specification for group actions on Uniform spaces} 
\author[Abdul Gaffar Khan, Pramod Kumar Das, Tarun Das]{Abdul Gaffar Khan$^{1}$, Pramod Kumar Das$^{2}$ and Tarun Das$^{1}$}                 
\keywords{Specification, Topological Entropy, Bowen Entropy, Uniform Spaces \vspace*{0.08cm}\\ 
\vspace*{0.01cm}
\Letter{Tarun Das} \\
\vspace*{0.08cm}
tarukd@gmail.com \\
\vspace*{0.01cm}
Abdul Gaffar Khan \\
\vspace*{0.08cm}
gaffarkhan18@gmail.com\\
\vspace*{0.01cm}
Pramod Kumar Das\\
\vspace*{0.08cm}
pramodkumar.das@nmims.edu\\
\vspace*{0.01cm}
\textit{$^{1}$Department of Mathematics, Faculty of Mathematical Sciences, University of Delhi, Delhi, India.}\\ 
\hspace*{0.11cm}\textit{$^{2}$School of Mathematical Sciences, Narsee Monjee Institute of Management Studies, Vile Parle, Mumbai-400056,  India.} 
}

\begin{abstract}
We extend specification and periodic specification to finitely generated group actions on uniform spaces using a concept of specification point. We prove that certain group actions having two distinct specification points have positive entropy. We further prove that if a group containing an infinite order element acts on an infinite Hausdorff uniform space and the action possesses periodic specification, then it is Devaney chaotic. 
\end{abstract}
\maketitle 

\section{Introduction}
One of the most important and extensively studied variant of shadowing in discrete topological dynamics is known as specification. The motivation for this concept comes from the wish to trace simultaneously, finite number of finite pieces of orbits by one periodic orbit. Bowen proved \cite{B} the usefulness of specification by studying the distribution of periodic points in the phase space. The periodicity of the tracing point made this notion to be popularly known as periodic specification among many other variants. The connection of this notion with the notion of chaos was established \cite{S} much later than its introduction \cite{B} by Bowen. One can see that most of the studies regarding this important variant of shadowing are on compact metric spaces. An interested reader can find a significant amount of literature in \cite{N}. 
\medskip

In recent years, a group of mathematicians from all over the world working to understand the validity of several important results for systems with non-compact, non-metrizable phase space. Such study gained desirable attention with the publication of \cite{DLRW} in which authors proved Walters' stability theorem and a similar version of Smale's spectral decomposition theorem for homeomorphisms on uniform spaces. In \cite{DD, SDD, SDD2016} authors studied periodic specification, weak specification and pseudo orbital specification for uniformly continuous maps on uniform spaces. At the same time, mathematicians started investigating shadowing, specification and their variants for group and semigroup actions on uniform spaces \cite{SC, DD2019, RV, S2019}.   
\medskip

One of the aim of this paper is to prove that periodic specification for certain group actions on infinite Hausdorff uniform space is Devaney chaotic. Another aim is to formulate the notion of Bowen entropy for finitely generated group actions on uniform spaces and then relate the positivity of entropy with the presence of specification.      
\medskip

In language of physics, entropy is a thermodynamical property which represents the amount of system's thermal energy lacks to convert into mechanical work. Mathematically, it is a non-negative extended real number representing amount of information obtained by performing an experiment repeatedly. If entropy is positive, then the system is chaotic. On the other hand, a deterministic system does have zero entropy. The classification problem of dynamical systems motivated mathematicians to introduce \cite{AKM} the notion of topological entropy for continuous surjective map on compact topological space similarly as the measure theoretic entropy \cite{W} for measure preserving transformations of measure spaces. Bowen extended \cite{W} this notion to uniformly continuous map on metric spaces. In \cite{H}, B. M. Hood showed that Bowen's definition is already suited for uniform spaces. 

\section{Preliminaries}

A uniform space is a pair $(X,\mathcal{U})$, where $X$ is a non-empty set and $\mathcal{U}$ is a collection of subsets of $X\times X$ satisfying the following properties. 
\begin{enumerate}
\item[(1)] Every $D\in\mathcal{U}$ contains $\Delta(X)$.
\item[(2)] If $D\in\mathcal{U}$ and $E\supset D$, then $E\in\mathcal{U}$.
\item[(3)] If $D,D'\in\mathcal{U}$, then $D\cap D'\in\mathcal{U}$.
\item[(4)] If $D\in\mathcal{U}$, then $D^{-1}\in\mathcal{U}$. 
\item[(5)] For every $D\in\mathcal{U}$ there is a symmetric $D'\in\mathcal{U}$ such that $D'\circ D'\subset D$.
\end{enumerate} 
\medskip

Where $\Delta(X)=\lbrace (x,x)\mid x\in X\rbrace$, $D^{-1}=\lbrace (y,x)|(x,y)\in D\rbrace$ and $D'\circ D'=\lbrace(x,y)\in X\times X |$ there is $z\in X$ satisfying $(x,z)\in D$ and $(z,y)\in D'\rbrace$. 
\medskip 

If $(X,\mathcal{U})$ is a uniform space, then we can generate a topology on $X$ by characterizing that a subset $Y\subset X$ is open if and only if for each $x\in Y$ there is an entourage $U\in\mathcal{U}$ such that the cross section $U[x]=\lbrace y\in X\mid (x,y)\in U\rbrace\subset Y$. Thus, for any point $x\in X$ and any neighbourhood $\mathfrak{G}$ of $x$, one can find $U\in\mathcal{U}$ such that $U[x]\subset \mathfrak{G}$. The members of $\mathcal{U}$ are called entourages. An entourage $D\in\mathcal{U}$ is said to be symmetric if $D=D^{-1}$. Observe that for any $U\in\mathcal{U}$, the entourage $U\cap U^{-1}$ is symmetric and $U^{n}=\lbrace (x,y)\mid z_{0}=x,z_{1},...,z_{n}=y \in X$ such that $(z_{i-1},z_{i})\in U$ for all $i\in \lbrace 1,2,...,n\rbrace \rbrace$. The set of all symmetric entourages and the set of all open symmetric entourages are denoted by $\mathcal{U}^s$ and $\mathcal{U}^0$ respectively. Two uniformities $\mathcal{U}$ and $\mathcal{V}$ on $X$ are called uniformly equivalent if identity maps $I_{1}: (X,\mathcal{U})\rightarrow (X,\mathcal{V})$ and $I_{2}: (X,\mathcal{V})\rightarrow (X,\mathcal{U})$ are uniformly continuous.  
\medskip
  
A non-empty family of non-empty subsets of $X$ is a filter base if the intersection of every pair of members of the family contains another member of the family. Observe that both $\mathcal{U}^s$ and $\mathcal{U}^0$ are filter bases for the uniformity $\mathcal{U}$. A directed set is a non-empty set $D$ together with a reflexive and transitive binary relation $\leq$ such that for any $x, y\in D$ there exists $z\in D$ satisfying $x\leq z$ and $y\leq z$. Observe that both $\mathcal{U}^s$ and $\mathcal{U}^0$ are directed sets with reverse set inclusion as the binary relation. A net in $X$ is a function $f$ from a directed set $D$ into $X$. Let $A$ be a set and $\leq$ be a binary relation on $A$. Then, a subset $B\subset A$ is said to be cofinal in $A$ if for every $a\in A$, there is $b\in B$ such that $a\leq b$. It is easy to see that $\mathcal{U}^0$ is cofinal in $\mathcal{U}^s$.  
\medskip

Let $(X,\mathcal{U})$ be a uniform space \cite{J} and $G$ be a finitely generated group. A map $\Phi:G\times X\rightarrow X$ is said to be a continuous action of $G$ on $X$ if the following conditions hold: 
\begin{enumerate}
\item[(i)] For each $g\in G$, the map $\Phi_g=\Phi(g,.)$ is a homeomorphism. 
\item[(ii)] $\Phi_e(x)=x$ for all $x\in X$, where $e$ is the identity element of the group $G$. 
\item[(iii)] $\Phi_{g_1g_2}(x)=\Phi_{g_1}(\Phi_{g_2}(x))$ for all $x\in X$ and $g_1,g_2\in G$. 
\end{enumerate} 
\medskip

An action $\Phi$ is said to be uniformly continuous if for each $g\in G$, $\Phi_g$ is a uniform equivalence (Both $\Phi_g$ and $\Phi_{g}^{-1}$ are uniformly continuous). We denote the set of all uniformly continuous actions by $Act(G,X)$. If $\Phi\in Act(G,X)$ and $\Psi\in Act(G,Y)$, then the diagonal product $\Phi\times \Psi\in Act(G, X\times Y)$ is defined as $(\Phi\times \Psi)_{g} (x,y) =(\Phi_{g}(x), \Psi_{g}(y))$.   
\medskip

Let $G_{1}=\lbrace g_{i}\mid 1\leq i\leq m\rbrace$ be a finite symmetric generating set for $G$. We can write $G=\cup_{n\geq 0} G_{n}$, where $G_{0}=\lbrace e\rbrace$ and for $n\geq 2$, $g\in G_{n}$ if and only if $g= g_{i_{n}}g_{i_{n-1}}...g_{i_{1}}$ with $g_{i_{j}}\in G_{1}$ for $1\leq j\leq n$. So, $G_n$ contains every element of length at most $n$ with respect to the generating set $G_{1}$. Note that if $G$ is countably infinite, then $G_{k}\setminus G_{k-1}\neq\phi$ for all $k\in \field{N}$ and if $G$ is finite then there exists $k\in \field{N}$ such that $G_{k}\setminus G_{k-1}\neq\phi$ but $G_{n+1}\setminus G_{n}=\phi$ for all $n\geq k$. 
\medskip

Let $d^{G}$ be the word metric on $G$ given by $d^{G}(h,g)=$ inf $\lbrace n \geq 0\mid h^{-1}g\in G_{n}\rbrace$. For $A, B\subset G$, define $d^{G}(A, B) =$ inf $\lbrace d^{G}(a, b)\mid (a,b)\in A\times B\rbrace$. For convenience, we write $d^G(a,b)$ for $d^G(\lbrace a\rbrace, \lbrace b\rbrace)$. The Hausdorff distance between two subsets $A$ and $B$ of $G$ is defined as $d_{G}(A, B) =$ max $\lbrace $sup$_{a\in A}d^{G}(a, B),$ sup$_{b\in B}d^{G}(A, b) \rbrace$. 
\medskip 

The set of all compact subsets of $X$ is denoted by $\mathcal{K}(X)$. A subset $W\subset X$ is said to be $\Phi$-invariant if $\Phi_s(W)\subset W$ for all $s\in G_{1}$. Action $\Phi$ is said to be transitive if for any pair of non-empty open sets $U,V$ in $X$ there exists $g\in G$ such that $\Phi_{g}(U)\cap V\neq \phi$. 
\medskip

Two actions $\Phi\in Act(G,X)$ and $\Psi\in Act(G,Y)$ are said to be uniformly conjugate if there is a uniform equivalence $T:X\rightarrow Y$ such that $T\Phi_{g} = \Psi_{g}T$ for all $g\in G$. A property of an action $\Phi$ preserved under uniform conjugacy is said to be a uniform dynamical property. 

\section{Entropy and Specification for Group Actions}  
Let $n$ be a positive integer, $U\in \mathcal{U}^{s}$ and $\Phi\in Act(G,X)$. We say that a subset $\mathfrak{E}$ of $X$ is $(n,U)$-separated with respect to $\Phi$, if for each pair of distinct points $x, y$ in $\mathfrak{E}$ there exists $g\in G_{n}$ such that $(\Phi_{g}(x), \Phi_{g}(y))\notin U$. A subset $\mathfrak{E}$ of $X$ is $(n, U)$-spanning set for another subset $\mathfrak{F}$ of $X$ with respect to $\Phi$ if for each $x\in \mathfrak{F}$ there exists $y\in \mathfrak{E}$ such that $(\Phi_{g}(x), \Phi_{g}(y))\in U$ for all $g\in G_n$. 
\medskip

For $\mathfrak{K}\in\mathcal{K}(X)$, let $s_{n}(U,\mathfrak{K},\Phi)$ be the maximal cardinality of any $(n, U)$-separated set contained in $\mathfrak{K}$ and $r_{n}(U,\mathfrak{K},\Phi)$ be the minimal cardinality of any $(n, U)$-spanning set for $\mathfrak{K}$. Define $r_{\Phi}(U, \mathfrak{K})= \limsup\limits_{n\rightarrow \infty} \frac{1}{n} log\hspace*{0.1cm} r_{n}(U,\mathfrak{K}, \Phi)$ and $s_{\Phi}(U, \mathfrak{K}) = \limsup\limits_{n\rightarrow \infty} \frac{1}{n}log\hspace*{0.1cm} s_{n}(U,\mathfrak{K}, \Phi)$.
\medskip

\begin{Lemma} For $\Phi\in Act(G,X)$ and $\mathfrak{K}\in \mathcal{K}(X)$, the following statements are true. 
\begin{enumerate}
\item For $U, V\in\mathcal{U}^{s}$ with $V^{2}\subset U$, we have $r_{n}(U,\mathfrak{K}, \Phi)\leq s_{n}(U,\mathfrak{K}, \Phi)\leq r_{n}(V,\mathfrak{K}, \Phi)\leq s_n(V,\mathfrak{K},\Phi)$.
\item For $U\subset V$, we have $r_{\Phi}(U, \mathfrak{K})\geq r_{\Phi}(V, \mathfrak{K})$ and $s_{\Phi}(U, \mathfrak{K})\geq s_{\Phi}(V, \mathfrak{K})$.
\item $\lim_{U\in\mathcal{U}^{s}}(r_{\Phi}(U, \mathfrak{K})) = \lim_{U\in\mathcal{U}^{s}}(s_{\Phi}(U, \mathfrak{K})) = \lim_{U\in\mathcal{U}^{o}}(s_{\Phi}(U, \mathfrak{K})) = \lim_{U\in\mathcal{U}^{o}}(r_{\Phi}(U, \mathfrak{K}))$.  
\end{enumerate}
\label{3.1} 
\end{Lemma}

\begin{proof} (1) Let $\mathfrak{K}\in \mathcal{K}(X)$ and $\mathfrak{K'}$ be the maximal $(n,U)$-separated subset of $\mathfrak{K}$. Then, $\mathfrak{K'}$ is $(n,U)$-spanning set for $\mathfrak{K}$. Indeed, if there exists $x\in \mathfrak{K}\setminus\mathfrak{K'}$ such that $(\Phi_{g}(x),\Phi_{g}(y))\notin U$ for all $y\in \mathfrak{K'}$ and some $g\in G_{n}$, then $x$ lies in a separated set for $\mathfrak{K}$ and so in $\mathfrak{K'}$ due to maximality. Therefore, $r_{n}(U,\mathfrak{K},\Phi)\leq s_{n}(U,\mathfrak{K},\Phi)$. Now suppose that $\mathfrak{K''}$ is minimal $(n,V)$-spanning set for $\mathfrak{K}$. 
So for each $x\in \mathfrak{K}$ there exists $f(x)\in \mathfrak{K''}$ such that $(\Phi_{g}(x),\Phi_{g}(f(x)))\in V$ for all $g\in G_{n}$. If $f(x)=f(y)$, then $(\Phi_{g}(x),\Phi_{g}(y))\in V^{2}\subset U$ for all $g\in G_{n}$. Since $\mathfrak{K'}$ is $(n,U)$-separated, $f$ must be injective on $\mathfrak{K'}$. Therefore, $|\mathfrak{K''}| \geq |\mathfrak{K'}|$ and so $r_{n}(V,\mathfrak{K},\Phi)\geq s_{n}(U,\mathfrak{K},\Phi)$. 
\\
(2) Since this follows from the definition, it is left as an easy exercise. 
\\
(3) Since $\Phi$ is uniformly continuous, for $V\in\mathcal{U}^o$ there exists $W\in\mathcal{U}^{o}$ such that $(x,y)\in W$ implies $(\Phi_{g}(x),\Phi_{g}(y))\in V$ for all $g\in G_{n}$. Then, $r_{n}(V,\mathfrak{K},\Phi)$ is at most the number of $W$-neighbourhoods require to cover $\mathfrak{K}$, which is finite because $\mathfrak{K}$ is compact. Therefore, $r_{n}(U,\mathfrak{K},\Phi)$ and $s_{n}(U,\mathfrak{K},\Phi)$ are also finite. Being filter bases for $\mathcal{U}$, $\mathcal{U}^{s}$ and $\mathcal{U}^{o}$ are directed sets, $r_{\Phi}(U, \mathfrak{K})$ and $s_{\Phi}(U, \mathfrak{K})$ are nets in non-negative reals. Further since $\mathcal{U}^{o}$ is cofinal in $\mathcal{U}^{s}$, they give us subnets and so by (2) we get the result.
\end{proof}

\begin{Definition}
For $\Phi\in Act(G, X)$ and $\mathfrak{K}\in \mathcal{K}(X)$, set $h(G_{1},\Phi, \mathfrak{K},\mathcal{U}) = \lim \lbrace r_{\Phi}(U, \mathfrak{K}) \mid U\in\mathcal{U}^{s}\rbrace = \lim \lbrace s_{\Phi}(U, \mathfrak{K}) \mid U\in\mathcal{U}^s\rbrace=\lim \lbrace r_{\Phi}(U, \mathfrak{K}) \mid U\in\mathcal{U}^{o}\rbrace=\lim \lbrace s_{\Phi}(U, \mathfrak{K}) \mid U\in\mathcal{U}^{o}\rbrace$  and $h(G_{1},\Phi, \mathcal{U}) = sup\lbrace h(G_{1},\Phi, \mathfrak{K}, \mathcal{U}) \mid \mathfrak{K} \in \mathcal{K}(X) \rbrace$. The number $h(G_{1},\Phi, \mathcal{U}) $ is called the entropy of $\Phi$ with respect to $\mathcal{U}$ and generator $G_{1}$ of $G$.  
\label{3.2}
\end{Definition}

\begin{Remark}
Let $\Phi\in Act(G,X)$ and $G_{1}$ be the generator of $G$. Then for each $s\in G_{1}$, the entropy of $\Phi_{s}$ is less than or equal to the entropy of $\Phi$ with respect to the generator $G_{1}$.   
\label{3.3}
\end{Remark}

\begin{theorem}
Let $\Phi\in Act(G,X)$. If $\mathcal{U}$ and $\mathcal{V}$ are uniformly equivalent uniformities, then $h(G_{1},\Phi, \mathcal{U}) = h(G_{1},\Phi, \mathcal{V})$. 
\label{3.4}
\end{theorem}
\begin{proof}
Let $U\in\mathcal{U}^{s}$, then there exists $V\in\mathcal{V}^{s}$ such that whenever $(x,y)\in V$, we have $(x,y)\in U$.  Choose $W\in\mathcal{U}^{s}$ such that whenever $(x,y)\in W$, we have $(x,y)\in V$. If $\mathfrak{K}\in \mathcal{K}(X)$, then $r_{\Phi}(U,\mathfrak{K}) \leq r_{\Phi} (V,\mathfrak{K}) \leq r_{\Phi}(W,\mathfrak{K})$. Then by the definition of entropy, we can conclude that $h(G_{1},\Phi, \mathcal{U}) = h(G_{1},\Phi, \mathcal{V})$. 
\end{proof}

\begin{theorem}
If $\Phi\in Act(G,X)$ and $\Psi\in Act(G, Y)$ are uniformly conjugate, then $h(G_{1},\Phi, \mathcal{U}) = h(G_{1},\Psi, \mathcal{V}) $. 
\label{3.5}
\end{theorem}
\begin{proof}
Let $f:(X,\mathcal{U})\rightarrow (Y,\mathcal{V})$ be a uniform conjugacy. By uniform continuity of $f$, for every $V\in\mathcal{V}$ there exists $U\in\mathcal{U}$ such that if $(x,y)\in U$ then $(f(x),f(y))\in V$. Let $n\in \field{N}$ and $\mathfrak{A}$ is $(n, U)$-spanning set for compact set $\mathfrak{K}$ with respect to $\Phi$. 
Then $f(\mathfrak{K})$ is compact and $f(\mathfrak{A})$ is $(n, V)$-spanning for $f(\mathfrak{K})$ with respect to $\Psi$. 
Since $|f(\mathfrak{A})|=|\mathfrak{A}|$, we have $r_{n}(V,f(\mathfrak{K}),\Psi)\leq r_{n}(U,\mathfrak{K},\Phi) $.
Therefore, $h(G_{1},\Psi, f(\mathfrak{K}),\mathcal{V})\leq h(G_{1},\Phi, \mathfrak{K},\mathcal{U})$.
Since there is one to one correspondence between compact subsets of $X$ and $Y$, we have  
$h(G_{1},\Psi, \mathcal{V})=$ $sup \lbrace h(G_{1},\Psi, \mathfrak{K},\mathcal{V}) :\mathfrak{K}\in \mathcal{K}(Y)\rbrace$ $=$ $sup \lbrace h(G_{1},\Psi, f(\mathfrak{K}), \mathcal{V}) :\mathfrak{K}\in \mathcal{K}(X)\rbrace \leq$ $sup \lbrace h(G_{1},\Phi,$ $ \mathfrak{K},\mathcal{U}) :\mathfrak{K}\in \mathfrak{K}(X)\rbrace$ $=h(G_{1},\Phi, \mathcal{U})$. Thus, $h(G_{1},\Psi, \mathcal{V})\leq h(G_{1},\Phi, \mathcal{U})$. 
Similarly, one can show that $h(G_{1},\Phi, \mathcal{U}) \leq h(G_{1},\Psi, \mathcal{V})$.
\end{proof}

\begin{Example}
Let $X = \mathbb{R}$ be with natural uniformity generated by the euclidean metric and $\Phi\in Act(G, X)$ an equicontinuous action. Then for every $\epsilon >0$ there exists $\delta_{\epsilon} >0$ such that $d(x,y)< \delta_{\epsilon}$ implies $ d(\Phi_{g}(x), \Phi_{g}(y))<\epsilon$ for all $g\in G$. Let $\mathfrak{K}\in \mathcal{K}(X)$ and $\mathfrak{A}$ be the maximal $(n,\epsilon)$-separated subset of $\mathfrak{K}$. Therefore by equicontinuity, if $x,y\in \mathfrak{A}$ then $d(x,y)> \delta_{\epsilon}$.  Thus, $s_{n}(\mathfrak{K},\epsilon ) \leq diam(\mathfrak{K})/\delta_{\epsilon}$ and so $s_{\Phi}(\epsilon,\mathfrak{K})=0$. Since $\epsilon$ was chosen arbitrary, we conclude that entropy of any equicontinuous action on $X$ is zero. 
\end{Example}

\begin{Definition}
Let $(X,\mathcal{U})$ be a uniform space and $\Phi\in Act(G, X)$. Then $z$ is said to be a specification point of $\Phi$
if for every $U\in\mathcal{U}^s$, there exists an integer $c(U_{z}) > 0$ such that for any $k\in\field{N}$, any finite family $(\Lambda_i)_{i=1}^{k}$ of subsets of $G$ with $d_{G}(\Lambda_i, \Lambda_j) > c(U_{z})$ for $i\neq j$ and any collection of points $(x_i)_{i= 1}^{k}$ with $x_{1} = z$, there exists a tracing point $x\in X$ such that $(\Phi_{g_{i}}(x), \Phi_{g_{i}}(x_{i}))\in U$ for all $g_i\in \Lambda_i$, $1\leq i\leq k$. If $x$ is a periodic point then we say that $z$ is a periodic specification point, where a point $x$ is said to be periodic if its orbit under $\Phi$ is finite. We say that $\Phi$ has specification if each point $z\in X$ is a specification point with a common specification integer $c(U_z)$ and that $\Phi$ has periodic specification if each point $z\in X$ is a periodic specification point with a common specification integer $c(U_z)$. 
\label{4.1}
\end{Definition}

\begin{theorem}
Specification and specification point does not depend on the choice of generator.
\end{theorem}
\begin{proof}
Let $G_{1}=\lbrace g_{i} : 1\leq i\leq p \rbrace$ and $H_{1}=\lbrace h_{i} : 1\leq i \leq q \rbrace$ be two generators of $G$. Then choose $N\in \field{N}$ such that every element of $G_{1}$ can be written in terms of elements of $H_{1}$ with length at most $N$. Suppose that $\Phi$ has specification with respect to the generator $G_{1}$. To show that $\Phi$ has specification with respect to the generator $H_1$. Let $U\in \mathcal{U}^{s}$ and $c(U)$ be a specification integer for $\Phi$ with respect to $G_1$. Let $(\Lambda_{i})_{i=1}^{k}$ be a finite family of subsets of $G$ with $d_{(G, H_{1})}(\Lambda_{i}, \Lambda_{j}) > Nc(U)$ for $i\neq j$ and $(x_i)_{i= 1}^{k}$ be points in $X$. Then $d_{(G, G_{1})}(\Lambda_{i}, \Lambda_{j})> c(U)$ for $i\neq j$ and hence, there exists $x\in X$ such that $(\Phi_{g_{i}}(x), \Phi_{g_{i}}(x_{i}))\in U$ for all $g_i\in \Lambda_{i}$ and $1\leq i\leq k$. Thus, $\Phi$ has specification with respect to the generator $H_{1}$. Similarly, one can prove that $z$ is a specification point for $\Phi$ with respect to $G_1$ implies that it is a specification point for $\Phi$ with respect to $H_1$.     
\end{proof} 

Recall that, a continuous map $h : X\rightarrow X$ is said to have specification if for every $U\in \mathcal{U}^{s}$ there exists an integer $p(U)\geq 1$ such that for each $k \geq 1$, any points $x_1, . . . , x_k$, and any sequence of positive integers $n_1, . . . ,n_k$ and $p_1, . . . , p_k$ with $p_i \geq p(U)$ there exists a point $x$ in $X$ such that $(h^{j}(x), h^{j}(x_{1}))\in U$ for all $0\leq j\leq n_{1}$ and $(h^{j+n_{1}+p_{1}+... + n_{i-1}+p_{i-1}(x)}, h^{j}(x_{1}))\in U$ for every $0\leq j\leq n_{i}$ and  $2\leq i\leq k$. 

\begin{theorem}
If $\Phi\in Act(G, X)$ has specification, then for each infinite order element $g\in G$, $\Phi_{g}$ has specification.   
\label{4.2}
\end{theorem}
\begin{proof}
Let $U\in \mathcal{U}^{s}$ and $c(U)$ is a specification integer for $\Phi$. Choose $k\in \field{N}$, points $\lbrace x_{1},. . . , x_{k}\rbrace$ in $X$, positive integers $n_{1}, . . .,n _{k}$ and $p_{1},. . . ,p_{k}$ such that $p_{i}\geq c(U)+1$ for every $p_{i}$ with $n_{0}= p_{0}=0$. For an infinite order element $g\in G$, we construct a finite generating set $G_{1}$ containing $g$. Clearly, $\Lambda_{i}=\lbrace g^{j} : \sum_{m=0}^{i-1}(p_{m}+n_{m})\leq j \leq n_{i}+ \sum_{m=0}^{i-1}(p_{m}+n_{m})\rbrace$ is finite for each $i=1,...,k$. In fact, $d_{G}(\Lambda_{i},\Lambda_{j})> c(U)$ for $i\neq j$, as $g$ has infinite order. Let $x_{j}^{*} = g^{-\sum_{m=0}^{j-1}}(x_{j})$ for $1\leq j\leq k$. By the specification of $\Phi$ there exists $x\in X$ such that $(\Phi_{g_{i}}(x), \Phi_{g_{i}}(x_{i}^{*}))\in U$ for all $1\leq i \leq k$ and $g_{i}\in \Lambda_{i}$, which is same as saying that $((\Phi_{g})^{j}(x), (\Phi_{g})^{j}(x_{1}))\in U$ for all $0\leq j\leq n_{1}$ and $((\Phi_{g})^{j+n_{1}+p_{1}+... + n_{i-1}+p_{i-1}(x)}, (\Phi_{g})^{j}(x_{1}))\in U$ for all $0\leq j\leq n_{i}$ and $2\leq i\leq k$. Therefore, $\Phi_{g}$ has specification.  
\end{proof}

\begin{theorem}
Two actions $\Phi\in Act(G,(X,\mathcal{U}))$ and $\Psi\in Act(G,(Y,\mathcal{V}))$ have specification (periodic specification) if and only if $\Phi\times \Psi$ the diagonal action has specification (periodic specification). 
\label{4.4}
\end{theorem}
\begin{proof}
Let us denote the diagonal product of $\Phi$ and $\Psi$ by $\varphi$, product uniformity $\mathcal{W}= \mathcal{U}\times \mathcal{V}$. Let $W\in \mathcal{W}^{s}$. Set $U= \lbrace (x_{1},x_{2})\in X\times X : (x_{1},y_{1},x_{2},y_{2})\in W$ for some $y_{1},y_{2}\in Y \rbrace$ and $V= \lbrace (y_{1},y_{2})\in Y\times Y : (x_{1},y_{1},x_{2},y_{2})\in W$ for some $x_{1},x_{2}\in X \rbrace$. Then $U\in \mathcal{U}^{s}$ and $V\in \mathcal{V}^{s}$. Let $c(U)$ and $c(V)$ be specification integers for $\Phi$ and $\Psi$ respectively. Set $c(UV) = max\lbrace c(U), c(V)\rbrace$. Let $\lbrace \Lambda_{i} \rbrace_{i=1}^{k}$ be a finite family of subsets of $G$ with $d_{G}(\Lambda_{i}, \Lambda_{j})> c(UV)$ for $i\neq j$, and $\lbrace (x_{i},y_{i})\rbrace_{i=1}^{k}$ be points in $X\times Y$. By specification of $\Phi$ and $\Psi$, we can choose $x\in X$ and $y\in Y$ such that $(\Phi_{g_{i}}(x), \Phi_{g_{i}}(x_{i}))\in U$ and $(\Psi_{g_{i}}(y), \Psi_{g_{i}}(y_{i}))\in V$ for all $g_{i}\in \Lambda_{i}$ and $1\leq i\leq k$. Thus we have $(\varphi_{g_{i}}(x,y), \varphi_{g_{i}}(x_{i},y_{i}))\in W$ for all $g_{i}\in \Lambda_{i}$ and $1\leq i\leq k$. Therefore, we conclude that $\varphi$ has specification. Moreover, if the points $x$ and $y$ are periodic points for $\Phi$ and $\Psi$, then $(x,y)$ is also a periodic point for $\varphi$. Thus, if $\Phi$ and $\Psi$ have periodic specification then $\varphi$ has periodic specification.   
\medskip

Conversely, suppose that $\varphi$ has specification. Let $U\in \mathcal{U}^{s}$ and $V\in \mathcal{V}^{s}$ then 
$W=\lbrace (x,y,x',y')\in (X\times Y)\times (X\times Y) : (x,x')\in U, (y,y')\in V\rbrace \in \mathcal{W}^{s}$. Choose a specification integer $c(W)$ for $\varphi$. Let $\lbrace \Lambda_{i} \rbrace_{i=1}^{k}$ be a finite family of subsets of $G$ with $d_{G}(\Lambda_{i}, \Lambda_{j})> c(W)$ for $i\neq j$, $\lbrace x_{i}\rbrace_{i=1}^{k}$, $\lbrace y_{i}\rbrace_{i=1}^{k}$ be points in $X$ and $Y$ respectively. By specification of $\varphi$ there exists $(x,y')\in X\times Y$ such that $(\varphi_{g_{i}}(x,y'), \varphi_{g_{i}}(x_{i},y_{i}))\in W$ for all $g_{i}\in \Lambda_{i}$ and $1\leq i\leq k$. Thus, $(\Phi_{g_{i}}(x), \Phi_{g_{i}}(x_{i}))\in U$ for all $g_{i}\in \Lambda_{i}$, $1\leq i\leq k$. Therefore, we conclude that $\Phi$ has specification. Similarly, one can prove that $\Psi$ has specification. Moreover if the point $(x, y)$ is a periodic point for $\varphi$, then $x$ and $y$ are also periodic points for $\Phi$ and $\Psi$ respectively. Therefore if $\varphi$ has periodic specification, then both $\Phi$ and $\Psi$ have periodic specification.   
\end{proof} 

\begin{theorem}
Specification, periodic specification are uniform dynamical property.  
\label{4.3}
\end{theorem}
\begin{proof}
Let $\gamma : Y\rightarrow X$ be a uniform conjugacy between $\Phi\in Act(G,(X,\mathcal{U}))$ and $\Psi\in Act(G, (Y, \mathcal{V}))$. So, $\gamma$ is a uniform equivalence satisfying $\Phi\circ \gamma = \gamma\circ \Psi$. Let $U\in \mathcal{U}^{s}$ and $\Gamma = \gamma\times \gamma$. By uniform continuity of $\Gamma$, there exists $V\in \mathcal{V}^{s}$ such that $\Gamma(V)\subset U$. Suppose that $\Psi$ has specification and $c(V)$ is a specification integer for $\Psi$. Let $\lbrace \Lambda_{i}\rbrace_{i=1}^k$ be a finite family of subsets of $G$ such that $d_{G}(\Lambda_{i},\Lambda_{j})> c(V)$ and $\lbrace x_{i}\rbrace_{i=1}^k$ be points in $X$. Since $\gamma$ is bijection, there exists unique points $\lbrace y_{1},. . . , y_{k}\rbrace$ in $Y$ such that $\gamma(y_{i})=x_{i}$ for $1\leq i\leq k$. By specification of $\Psi$, there exists $y\in Y$ such that $(\Psi_{g_{i}}(y), \Psi_{g_{i}}(y_{i}))\in V$ for all $g_{i}\in \Lambda_{i}$, $1\leq i\leq k$. If $\gamma(y)=z$, then $(\Psi_{g_{i}}(\gamma^{-1}z), \Psi_{g_{i}}(\gamma^{-1}x_{i}))= (\gamma^{-1}\Phi_{g_{i}}(z), \gamma^{-1} \Phi_{g_{i}}(x_{i}))\in V$. Therefore, $(\Phi_{g_{i}}(z),\Phi_{g_{i}}(x_{i}))\in U$ for all $g_{i}\in \Lambda_{i}$, $1\leq i\leq k$. Therefore, we conclude that $\Phi$ has specification. Converse follows similarly because $\gamma$ is a uniform equivalence. Moreover if the point $y$ is periodic for $\Psi$ then $z$ is also periodic for $\Phi$. Therefore, if $\Psi$ has periodic specification then $\Phi$ has periodic specification and vice versa.   
\end{proof}

\section{Specification implies Positive Entropy and Devaney Chaos}

Recall that a continuous action is said to be Devaney chaotic if it is transitive, admits dense set of periodic points and has sensitive dependence on initial condition. In this section, our first aim is to prove that on infinite Hausdorff uniform space certain group actions having periodic specification is Devaney chaotic. Using similar steps as in the proof of Theorem 2 \cite{SC}, one can prove that any transitive action admitting dense set of periodic points on infinite Hausdorff uniform space is sensitive. Therefore, it is sufficient to show that such group actions are transitive and has dense set of periodic points.  
\medskip

Recall that an action is said to have strong mixing property if for any pair of non-empty open sets $\mathfrak{U}$ and $\mathfrak{V}$ in $X$, cardinality of the set $G(\mathfrak{U},\mathfrak{V} ) = \lbrace g\in G$ : $\Phi_{g}\mathfrak{U} \cap \mathfrak{V}= \phi\rbrace$ is finite. 

\begin{theorem}
Let $G$ be an infinite order group and $\Phi\in Act(G,X)$. If $\Phi$ has specification, then $\Phi$ is strongly mixing and hence, transitive. 
\label{4.5}
\end{theorem}
\begin{proof} 
Let $\mathfrak{V}$ and $\mathfrak{W}$ be non-empty open subsets of $X$ with $v\in \mathfrak{V}$ and $w\in \mathfrak{W}$. 
Choose $U\in \mathcal{U}^{s}$ such that $U[v]\subset \mathfrak{V}$ and $U[w]\subset \mathfrak{W}$. Choose a specification integer $c(U)$ for $\Phi$, $h = g^{-1}$ for some $g\in G\setminus G_{c(U)+1}$, $\Lambda_{1} =\lbrace  e \rbrace$ and $\Lambda_{2} = G\setminus G_{c(U)+1}$. Set $x_{1}=v$ and $x_{2}= \Phi_{h}(w)$. By specification of $\Phi$ there exists $x\in X$ such that $(\Phi_{g_{i}}(x),\Phi_{g_{i}}(x_{i}))\in U$ for all $g_{i}\in \Lambda_{i}$, $1\leq i\leq 2$. Thus, $x\in \mathfrak{V} $ and $\Phi_{h^{-1}}(x)\in \mathfrak{W} $. Since $g$ was chosen arbitrary, we get that $\Phi_{g}\mathfrak{V}\cap \mathfrak{W}\neq \phi$ for all $g\in G\setminus G_{c(U)+1}$ and hence, $\Phi$ is strongly mixing. 
\end{proof}

\begin{theorem}
If $G$ contains an element of infinite order and $\Phi\in Act(G,X)$ has periodic specification, then it has dense set of periodic points.
\label{4.6}
\end{theorem}
\begin{proof}
Consider a finite symmetric generating set $G_{1}$ containing an infinite order element $s$. Let $x\in X$, $\mathfrak{V}$ be a neighbourhood of $x$. Let $U\in \mathcal{U}^{s}$ be such that $U[x]\subset \mathfrak{V}$. Choose a specification integer $c(U)$ for $\Phi$. Consider $\lbrace \Lambda_{i} = s^{c(U)(i-1)}\rbrace_{i=1}^{k}$ and $\lbrace x_{i}\rbrace_{i=1}^{k}$ with $x_{1}=x$. By periodic specification there exists a periodic point $x'\in X$ such that $(\Phi_{g_{i}}(x'), \Phi_{g_{i}}(x_{i}))\in U$ for all $g_{i}\in \Lambda_{i}$, $1\leq i\leq k$. In particular, $(x',x)\in U$. Since $U$ is symmetric, we have $x'\in \mathfrak{V}$. Hence, every open set in $X$ contains a periodic point which means the set of all periodic points of $\Phi$ is dense in $X$.  
\end{proof}

\begin{Corollary}
Let $G$ be a group containing an element of infinite order and $X$ be an infinite Hausdorff uniform space. If $\Phi\in Act(G,X)$ has periodic specification, then it is Devaney chaotic.
\label{4.7}
\end{Corollary}

\begin{Proposition}
Positivity of entropy of an action $\Phi\in Act(G,X)$ is independent of choice of generator. 
\end{Proposition}
\begin{proof}
Let $G_{1}=\lbrace g_{i} \mid 1\leq i\leq p \rbrace$ and $H_{1}=\lbrace h_{i} \mid 1\leq i \leq q \rbrace$ be two generators. Set $G=\cup_{n\geq 0} G_{n}$, where $G_{0}=\lbrace e\rbrace$ and $G=\cup_{n\geq 0} H_{n}$, where $H_{0}=\lbrace e\rbrace$. 
For $U\in \mathcal{U}^{s}$, $s_{n}^{G_{1}}(U,\mathfrak{K}, \Phi)$ and $s_{n}^{H_{1}}(U,\mathfrak{K}, \Phi)$ denotes the maximal cardinality of $(n,U)$-separated subset of any $\mathfrak{K}\in \mathcal{K}(X)$, with respect to the generators $G_{1}$ and $H_{1}$ respectively. Then we can choose $m_{1}\in \field{N}$ such that $G_{1}\subset H_{m_{1}}$ and thus we have $G_{n}\subset H_{nm_{1}}$ for all $n\in \field{N}$. It is easy to check that, for $\mathfrak{K}\in\mathcal{K}(X)$ and $U\in \mathcal{U}^{s}$, we have $s_{nm_{1}}^{H_{1}}(U,\mathfrak{K}, \Phi)\geq s_{n}^{G_{1}}(U,\mathfrak{K}, \Phi)$. Hence $\limsup\limits_{n\rightarrow \infty} \frac{1}{n}log\hspace*{0.1cm} s_{n}^{H_{1}}(U,\mathfrak{K}, \Phi) \geq \limsup\limits_{n\rightarrow \infty} \frac{1}{nm_{1}}log\hspace*{0.1cm} s_{nm_{1}}^{H_{1}}(U,\mathfrak{K}, \Phi) \geq \frac{1}{m_{1}}\limsup\limits_{n\rightarrow \infty} \frac{1}{n}log\hspace*{0.1cm} s_{n}^{G_{1}}(U,\mathfrak{K}, \Phi),$ implies $m_{1}s_{\Phi}^{H_{1}}(U,\mathfrak{K})\geq s_{\Phi}^{G_{1}}(U,\mathfrak{K})$. Therefore, $m_{1}h(H_{1},\Phi, \mathfrak{K},$ $\mathcal{U}) \geq h(G_{1},\Phi, \mathfrak{K},\mathcal{U}) $ for all $U\in \mathcal{U}^{s}$ and $\mathfrak{K}\in \mathcal{K}(X)$ and hence, $m_{1}h(H_{1},\Phi, \mathcal{U})\geq h(G_{1},\Phi, \mathcal{U})$. Similarly, we can choose $m_{2}\in \field{N}$ such that $m_{2}h(G_{1},\Phi, \mathcal{U})\geq h(H_{1},\Phi, \mathcal{U})$. Hence if entropy is positive with respect to $G_{1}$ then it is  positive with respect to $H_{1}$ and conversely.
\end{proof} 

\begin{theorem}
Let $(X,\mathcal{U})$ be a Hausdorff uniform space. If $G$ contains an element of infinite order and $\Phi\in Act(G,X)$ has two distinct specification points then entropy of $\Phi$ is positive. 
\label{4.8}
\end{theorem}
\begin{proof}
Consider a finite symmetric generating set $G_{1}$ containing an infinite order element $s$. Let $x,y\in X$ be two distinct specification points, $U\in \mathcal{U}^{s}$ such that $(x,y)\notin U^{2}$ and set $M = c(U) = max\lbrace c(U_{x}), c(U_{y})\rbrace$. Choose two $(n+1)$-tuples $(z_{1},. . . , z_{n+1})$ and $(z_{1}^{'},. . . ,z_{n+1}^{'})$ with $z_{1}=x$, $z_{1}^{'}=y$, $z_{i}, z_{i}^{'}\in \lbrace x,y\rbrace$ for all $2\leq i\leq (n+1)$ and $\Lambda_{i} = \lbrace s^{c(U)(i-1)}\rbrace$ for all $1\leq i \leq (n+1)$. Choose $z,z'\in X$ due to the specification at $x$ and $y$ respectively. First observe that $z\neq z'$. 
If $z=z'$ then $(\Phi_{g_{i}}(z),\Phi_{g_{i}}(z_{i}))\in U$ and $(\Phi_{g_{i}}(z),\Phi_{g_{i}}(z_{i}^{'}))\in U$ for all $g_{i}\in \Lambda_{i}$ and all $1\leq i\leq (n+1)$. For $i=1$ we get that $(z,z_{1})\in U$ and $(z,z_{1}^{'})\in U$ implies $(x,y)\in U^{2}$ which is a contradiction. Consider two $(n+2)$-tuples $(z_{1},. . . , z_{n+1}, z_{n+2})$ and $(z_{1}^{'},. . . ,z_{n+1}^{'}, z_{n+2}^{'})$ with $z_{1}=z_{1}^{'}\in \lbrace x,y\rbrace$, $z_{i}, z_{i}^{'}\in \lbrace x,y\rbrace$ for all $2\leq i\leq (n+1)$, $z_{n+2} = \Phi_{s^{-(c(U)(n+1))}}(x)$, $z_{n+2}^{'} = \Phi_{s^{-(c(U)(n+1))}}(y)$ and $\Lambda_{i} = \lbrace s^{c(U)(i-1)}\rbrace$ for all $1\leq i \leq (n+2)$. Using the similar arguments, we can choose distinct tracing points for these tuples. Therefore, for $(n+1)$-tuples we can choose a distinct tracing points due to specification at $x$ and $y$. Thus there are atleast $2^{n}$ , $(nM, U)$-separated points. Therefore, $h(G_{1},\Phi, \mathcal{U}) = sup\lbrace h(G_{1},\Phi, \mathfrak{K}, \mathcal{U}) : \mathfrak{K}\in \mathcal{K}(X) \rbrace \geq \lim \lbrace s_{\Phi}(U, \mathfrak{K}) : U\in \mathcal{U}^{s}\rbrace = \lim \lbrace \limsup\limits_{n\rightarrow \infty} \frac{1}{n}log\hspace*{0.1cm} s_{n}(U,\mathfrak{K}, \Phi) : U\in \mathcal{U}^{s}\rbrace \geq \limsup\limits_{n\rightarrow \infty} \frac{1}{n}log\hspace*{0.1cm} s_{n}(U,\mathfrak{K}, \Phi) \geq  \limsup\limits_{n\rightarrow \infty} \frac{1}{nM}log\hspace*{0.1cm} s_{nM}(U,\mathfrak{K}, \Phi) \geq  \limsup\limits_{n\rightarrow \infty} \frac{1}{nM} log\hspace*{0.1cm} 2^{n}= \frac{log\hspace*{0.1cm} 2}{M}>0$, which completes the proof. 
\end{proof}

\begin{Corollary}
Let $G$ be a group containing an element of infinite order and $X$ be a Hausdorff uniform space containing more than one point. If $\Phi\in Act(G,X)$ has specification, then the entropy of $\Phi$ is positive. 
\label{4.9}
\end{Corollary}

\begin{Example}
Let $G$ be any finite group of cardinality $\mathcal{C}$ and $X$ be arbitrary uniform space. If $\Phi\in Act(G, X)$, then $\Phi$ has specification. Now assume $X = \lbrace x_{n}\rbrace_{n=1}^{\infty} $, where $x_{n} =\sum_{i=1}^{n}(\frac{1}{i})$ with uniformity inherited from Euclidean metric. So it is a discrete uniformity. Let $D_{x}$ be an entourage such that $D_{x}[x]=\lbrace x\rbrace$. Let $\Phi$ be the trivial action. Then $\Phi$ has specification but does not have transitivity. Note that any compact subset $\mathfrak{K}$ is a set having finitely many elements. Thus we have $s_{n}(U, \mathfrak{K},\phi)\leq card(\mathfrak{K})$ for every $n\in \field{N}$ which implies $s_{\Phi}(U, \mathfrak{K}) = 0$ for every $U$ and $\mathfrak{K}\in \mathcal{K}(X)$. Hence, $\Phi$ has zero entropy. Therefore theorem \ref{4.5} and Corollary \ref{4.9} do not hold for finite groups.
\label{4.10}
\end{Example}

\begin{Example}
Let $G=\field{Z}^{2}$ with basis $\lbrace e_{1}=(1,0), e_{2}=(0,1), e_{3}=(-1,0), e_{4}=(0,-1)\rbrace$ and $X = \field{R}$ with uniformity generated by the Euclidean metric. Define $\Phi\in Act(G,X)$ by $\Phi_{e_{1}}(x) = x + 2$ and $\Phi_{e_{2}}(x) = x - 2$. 
Since $\Phi$ is not transitive, by Theorem \ref{4.5} $\Phi$ does not have specification. 
\label{4.11}
\end{Example}

\begin{Example}
Let $G=\field{Z}^{2}$ with generator $\lbrace e_{1}=(1,0), e_{2}=(0,1), e_{3}=(-1,0), e_{4}=(0,-1)\rbrace$ and $X = \lbrace x_{n}\rbrace_{n=1}^{\infty} $, where $x_{n} =\sum_{i=1}^{n}(\frac{1}{i})$. Suppose that $X$ has the discrete uniformity generated by the Euclidean metric. Let $D_{x}$ be an entourage such that $D_{x}[x]=\lbrace x\rbrace$. Let $\Phi$ be the trivial action. One can easily check that entropy of $\Phi$ is zero, $\Phi$ is not transitive and hence, does not have specification.   
\label{4.12}
\end{Example}

\begin{Example}
Let $G=\field{Z}^{2}$ with generator $\lbrace e_{1}=(1,0), e_{2}=(0,1), e_{3}=(-1,0), e_{4}=(0,-1)\rbrace$ and $X = \field{R}$ with uniformity generated by euclidean metric. Define $\Phi\in Act(G,X)$ by $\Phi_{e_{1}}(x) = 2x$ and $\Phi_{e_{2}}(x) = (1/2)x$. Since $\Phi_{e}(x)$ has entropy equal to $log\hspace*{0.1cm} 2$ \cite{SDD}, by Remark \ref{3.3} we get that $\Phi$ has entropy atleast $log\hspace*{0.1cm} 2$. Since $\Phi$ is not transitive, by Theorem \ref{4.5} $\Phi$ does not have specification. Hence, converse of Corollary \ref{4.9} is not true.   
\label{4.13}
\end{Example}

\begin{Example}
Let $\Phi\in Act(\mathbb{Z}, X)$ be an action generated by homeomorphism $f$ on $X$ i.e. $\Phi_{n}(x) = f^{n}(x)$. 
It is easy to see that Definition \ref{4.1} implies the classical definition of specification for homeomorphism.  Following example proves that the converse is not true.
Let $X_{i} = \lbrace 0, 1\rbrace$ be equipped with the discrete metric for all $i\in \mathbb{Z}$. Let $X = \prod_{i\in \mathbb{Z}}X_{i}$ be equipped with the  metric $D(x, y) = \sum_{i\in \mathbb{Z}}\frac{d(x_{i}, y_{i})}{2^{|i|}}$. 
Let $f$ be a left shift map on $X$ i.e $f(x_{n}) = x_{n+1}$, where $x = (x_{n})_{n\in \mathbb{Z}}$. Clearly, $f$ has specification \cite{S}. We claim that $\Phi$ does not have specification. Let $\delta = \frac{1}{2}$. 
If $d(x, y) < \delta$ then $(x)_{i} = (y)_{i}$ for all $-1\leq i\leq 1$ and if $d(x, y) < 2\delta$ then $(x)_{0} = (y)_{0}$. For each $j\in \mathbb{N}$, set $\Lambda_{1}^{j} = \lbrace 1, 1+1,..., 1+(j+1)\rbrace$ and $\Lambda_{2}^{j} = \lbrace 1+ (j+1), 1+(j+2),..., 1+(j+1) + (j+1)\rbrace$. For generating set $S = \lbrace 1, -1\rbrace$, it is easy to see that $d_{G}(\Lambda_{1}^{j}, \Lambda_{2}^{j}) = j+1 > j$ for all $j\in \mathbb{N}$. 
For each $j\in \mathbb{N}$, choose $x_{1}^{j}, x_{2}^{j}$ such that $(x_{1}^{j})_{1+(j+1)} \neq (x_{2}^{j})_{1+(j+1)}$. On contrary, choose an integer $c(\delta) = k > 0$ by specification of $\Phi$. 
Choose $x\in X$ be tracing point by specification property corresponding to $\lbrace \Lambda_{i}^{k}\rbrace_{i=1}^{2}$ and $\lbrace x_{i}^{k}\rbrace_{i=1}^{2}$. Then $d(f^{1+(k+1)}(x), f^{1+(k+1)}(x_{1}^{k})) < \delta$ and $d(f^{1+(k+1)}(x), f^{1+(k+1)}(x_{2}^{k})) < \delta$, which implies $d(f^{1+(k+1)}(x_{1}^{k}), f^{1+(k+1)}(x_{2}^{k})) < 2\delta = 1$. This holds only when $(x_{1}^{k})_{1+(k+1)} = (x_{2}^{k})_{1+(k+1)}$, which is a contradiction. Hence $\Phi$ does not have the specification property. 
\end{Example}

\textbf{Acknowledgements:} First author is supported by CSIR-Junior Research Fellowship (File No.-09/045(1558)/2018-EMR-I) of  Government of India.

\end{document}